\documentclass[preprint,12pt]{article}

\usepackage{amssymb}
\usepackage{mathrsfs}
\usepackage{amsfonts}
\usepackage{graphicx}
\usepackage{colortbl,dcolumn}
\usepackage{amsmath}
\usepackage{psfrag}
\usepackage{booktabs}
\usepackage{slashbox}
\usepackage{cite}
\usepackage{amsthm}

\numberwithin{equation}{section}
\usepackage[top=1.2in, bottom=1.2in, left=1in, right=1in]{geometry}

\newtheorem{theorem}{Theorem}[section]

\newtheorem{remark}{Remark}

\begin{document}
\title{Energy-conserving method for Stochastic Maxwell Equations with Multiplicative Noise}
       \author{
        Jialin Hong\footnotemark[2], Lihai Ji\footnotemark[3], Liying Zhang\footnotemark[1] and Jiaxiang Cai\footnotemark[4] \\
       {\small \footnotemark[2] Institute of Computational Mathematics and Scientific/Engineering Computing,}\\{\small Academy of Mathematics and Systems Science, Chinese Academy of Sciences, }\\
         {\small Beijing 100190, P.R.China }\\
          {\small \footnotemark[3] Institute of Applied Physics and Computational Mathematics,}\\
         {\small Beijing 100094, China}\\
         {\small \footnotemark[1] School of Mathematical Science, China University of Mining and Technology,}\\
         {\small Beijing 100083, China}\\
         {\small \footnotemark[4] School of Mathematical Science, Huaiyin Normal University,}\\
         {\small Huai'an, Jiangsu 210046, China}\\
          }
       \maketitle
       \footnotetext{\footnotemark[2] The first author is supported by the NNSFC (NO. 91130003, NO. 11021101 and NO. 11290142).\\
        \footnotemark[1]The third author is supported by the NNSFC (NO. 11301001, NO. 2013SQRL030ZD).}
        \footnotetext{\footnotemark[1]Corresponding author: lyzhang@lsec.cc.ac.cn}

       \begin{abstract}
          {\rm\small In this paper, it is shown that three-dimensional stochastic Maxwell equations with multiplicative noise are stochastic Hamiltonian partial differential equations possessing a geometric structure (i.e. stochastic mutli-symplectic conservation law), and the energy of system is a conservative quantity almost surely. We propose a stochastic multi-symplectic energy-conserving method for the equations by using the wavelet collocation method in space and stochastic symplectic method in time. Numerical experiments are performed to verify the excellent abilities
of the proposed method in providing accurate solution and preserving energy. The mean square convergence result of the method in temporal direction is tested
numerically, and numerical comparisons with finite difference method are also investigated.}\\


\end{abstract}
\section{Introduction}\label{1.1}
In this paper, we consider 3D stochastic Maxwell equations with multiplicative noise
\begin{equation}\label{stochastic maxwell equations}
\begin{split}
d\textbf{E}&=\nabla \times \textbf{H}dt-\lambda \textbf{H}\circ dW(t),\\[3mm]
d\textbf{H}&=-\nabla \times \textbf{E}dt+\lambda \textbf{E}\circ dW(t),
\end{split}
\end{equation}
where $t\in[0,T],~\textbf{x}=(x,y,z)\in\Theta\subset \mathbb{R}^{3}$, and $\Theta$ is a bounded and simply connected domain with smooth boundary $\partial\Theta.$ We employ the perfectly electric conducting (PEC) boundary condition
\begin{equation}\label{PEC}
\textbf{E}\times \textbf{n}=\textbf{0}
\end{equation}
on $(0, T] \times \partial\Theta,$ where
$\textbf{n}$ is the unit outward normal of $\partial\Theta$.
The above system is understood in the Stratonovich setting and the symbol $\circ$ stands for the Stratonovich product. Here, $\lambda\geq 0$ measures the size
of the noise and  $W$ is a $Q$-Wiener process defined on a given probability space $(\Omega,\mathcal{F},P,\{{\mathcal F}_{t}\}_{t\in[0,T]})$, with values in the Hilbert space ${\mathbb L}^2(\Theta)$, which is a space of square integrable real-valued functions. Let $\{e_m\}_{m\in \mathbb{N}}$ be an orthonormal basis of ${\mathbb L}^2(\Theta)$ consisting of eigenvectors of a symmetric, nonnegative and finite trace operator $Q$, i.e., $Tr(Q)<\infty$ and $Qe_m=\eta_me_m$. Then there exists a sequence of independent real-valued Brownian motions $\{\beta_m\}_{m\in \mathbb{N}}$ such that
$W(t,\textbf{x},\omega)=\sum\limits_{m=0}^{\infty}\sqrt{\eta_m}\beta_{m}(t,\omega)e_{m}(\textbf{x}),~~ t\geq 0, \textbf{x}\in\Theta,\omega\in \Omega.$

3D stochastic Maxwell equations with multiplicative noise play an important role in many scientific fields, especially in stochastic electromagnetism and statistical radiophysics \cite{Chauvi,Haier,RefKurt,Rytov}. We refer interested readers to \cite{RefLS} for the well-posedness of equations \eqref{stochastic maxwell equations}. By using infinite dimensional It\^{o} formula, it is straightforward to show that the $L^2(\Theta)$-norm of the solution is a constant almost surely (for more details see section 2), i.e.,
\begin{equation}
\begin{split}
\int_{\Theta}(|\mathbf{E}(\textbf{x},t,\omega)|^{2}+|\mathbf{H}(\textbf{x},t,\omega)|^{2})d\Theta=constant.
\end{split}
\end{equation}

There have been a lot of ongoing research activities in energy-conserving numerical methods for deterministic Maxwell equations, and various methods have been proposed in the literatures \cite{Chena,Chenb,Kong,Zhu11}. However, in the stochastic setting, we are only aware of the numerical schemes proposed in \cite{Cohen,Faou,Hong11,Misawa} for stochastic Hamiltonian ODEs. The authors in \cite{CZ,Refhai} proposed stochastic multi-symplectic methods for stochastic Maxwell equations with additive noise, which have the merits
of preserving the discrete stochastic multi-symplectic conservation law and stochastic energy dissipative
properties.

To the best of our knowledge, there has been no reference considering this aspect for stochastic Maxwell equations with multiplicative noise till now. In addition,
numerical methods preserving the structure characteristics of equations should be much better in preservation of physical properties and have better stability in numerical computation. Generally speaking, this kind of numerical methods constructed by finite difference techniques is completely implicit for non-separable stochastic system, and demands substantial computational cost.
In this paper, we use the idea of wavelet collocation method to get an efficient, multi-symplectic and energy-conserving numerical method.
By using this method, we obtain a
system of algebraic equations with a sparse
differentiation matrix, leading to a numerical algorithm
of reduced computational cost. Several numerical examples are presented to show the good behaviors of the proposed method by making comparison with a standard finite difference method. And the mean square convergence order of the proposed numerical method in temporal direction is shown.

The rest of this paper is organized as follows. In section 2, we present some geometric and physical properties of 3D stochastic Maxwell
equations with multiplicative noise. It is shown that the phase flow of equations preserves the stochastic multi-symplectic structure of phase space, and the equations possess energy conservation law.
In section 3, we propose a numerical method and show that the method preserves  discrete energy conservation law and the discrete stochastic multi-symplectic conservation law. In section 4, numerical experiments are performed to testify the effectiveness of the method. Concluding remarks are presented in section 5.
\section{Stochastic Maxwell equations}
In this section, we will present some preliminary results of 3D stochastic Maxwell equations \eqref{stochastic maxwell equations}. For simplicity in notations, we just consider the case that Wiener process $W(t,\textbf{x},\omega)$ is applied in one dimensional spatial direction.
\subsection{Stochastic multi-symplectic structure}
\label{SMSHMEs}
A stochastic partial differential equation is
called a stochastic Hamiltonian partial differential equation if it can be
written in the form \cite{RefY}
\begin{equation}\label{10}
Md_{t}u+Ku_{x}dt=\nabla S_{1}(u)dt+\nabla S_{2}(u)\circ
dW(t),~~u\in \mathbb{R}^{d},
\end{equation}
where $M$ and $K$ are skew-symmetric matrices, and $S_{1}$ and
$S_{2}$ are real smooth functions of the variable $u$.
Stochastic Maxwell equations with multiplicative noise (\ref{stochastic maxwell equations}) can be written as follows

\begin{equation}\label{12}
Md_{t}u+K_{1}u_{x}dt+K_{2}u_{y}dt+K_{3}u_{z}dt=\nabla_{u}S(u)\circ
dW,~~u\in \mathbb{R}^{6},
\end{equation}
where $
u=\big(H_{1},H_{2},H_{3},E_{1},E_{2},E_{3}\big)^{T}$, $
S(u)=\frac{\lambda}{2}(|\textbf{E}|^{2}+|\textbf{H}|^{2})$ and
\begin{equation*}
M=\left(\begin{array}{ccccccc}
0&-I_{3\times3}\\[1mm]
I_{3\times3}&0\\[1mm]
\end{array}\right),
K_{i}=\left(\begin{array}{cccccccc}
\mathcal{D}_{i}&0\\[1mm]
0&\mathcal{D}_{i}\\[1mm]
\end{array}
\right),~~i=1, 2, 3.
\end{equation*}
The sub-matrix $I_{3\times3}$ is a $3\times3$ identity matrix and
\begin{small}
$$
\mathcal{D}_{1}=\left(\begin{array}{ccccccc}
0&0&0\\[1mm]
0&0&-1\\[1mm]
0&~1&0\\[1mm]
\end{array}\right),
\mathcal{D}_{2}=\left(\begin{array}{ccccccc}
0&0&~1\\[1mm]
0&0&0\\[1mm]
-1&0&0\\[1mm]
\end{array}\right),
\mathcal{D}_{3}=\left(\begin{array}{ccccccc}
0&-1&0\\[1mm]
~1&0&0\\[1mm]
0&0&0\\[1mm]
\end{array}\right).
$$
\end{small}
Similarly to the proof of Theorem 2.2 in \cite{RefY}, we have the following result.
\begin{theorem}\label{them4}
System \eqref{12} possesses the stochastic multi-symplectic conservation law locally
\begin{equation*}\label{13}
d_{t}\omega+\partial_{x}\kappa_{1}dt+\partial_{y}\kappa_{2}dt+\partial_{z}\kappa_{3}dt=0, ~a.s.,
\end{equation*}
i.e.,
\begin{eqnarray}
\int_{z_{0}}^{z_{1}}\int_{y_{0}}^{y_{1}}\int_{x_{0}}^{x_{1}}\omega(t_{1},x,y,z)dxdydz+
\int_{z_{0}}^{z_{1}}\int_{y_{0}}^{y_{1}}\int_{t_{0}}^{t_{1}}\kappa_{1}(t,x_{1},y,z)dtdydz\nonumber\\
+\int_{z_{0}}^{z_{1}}\int_{x_{0}}^{x_{1}}\int_{t_{0}}^{t_{1}}\kappa_{2}(t,x,y_{1},z)dtdxdz+
\int_{y_{0}}^{y_{1}}\int_{x_{0}}^{x_{1}}\int_{t_{0}}^{t_{1}}\kappa_{3}(t,x,y,z_{1})dtdxdy\nonumber\\
=\int_{z_{0}}^{z_{1}}\int_{y_{0}}^{y_{1}}\int_{x_{0}}^{x_{1}}\omega(t_{0},x,y,z)dxdydz+
\int_{z_{0}}^{z_{1}}\int_{y_{0}}^{y_{1}}\int_{t_{0}}^{t_{1}}\kappa_{1}(t,x_{0},y,z)dtdydz\nonumber\\
+\int_{z_{0}}^{z_{1}}\int_{x_{0}}^{x_{1}}\int_{t_{0}}^{t_{1}}\kappa_{2}(t,x,y_{0},z)dtdxdz+
\int_{y_{0}}^{y_{1}}\int_{x_{0}}^{x_{1}}\int_{t_{0}}^{t_{1}}\kappa_{3}(t,x,y,z_{0})dtdxdy,\nonumber
\end{eqnarray}
where $\omega(t,x,y,z)=\frac{1}{2}du\wedge Mdu$,
$\kappa_{i}(t,x,y,z)=\frac{1}{2}du\wedge K_{i}du$ are
the differential 2-forms associated with the skew-symmetric matrices
$M$ and $K_{i}~(i=1,2,3)$, respectively, and
$(t_{0},t_{1})\times(x_{0},x_{1})\times(y_{0},y_{1})\times(z_{0},z_{1})$
is the local definition domain of $u(t,x,y,z)$.
\end{theorem}

\subsection{\label{sec:energy conservation}Energy conservation law}
As is well known, the deterministic Maxwell equations have the following invariant \cite{Chena}
\begin{equation}\label{deter}
\begin{split}
\int_{\Theta}(|\textbf{E}(x,y,z,t)|^{2}+|\textbf{H}(x,y,z,t)|^{2})d\Theta=constant.
\end{split}
\end{equation}

This invariant is also called Poynting theorem in electromagnetism and can be easily verified. Similarly, based on the infinite dimensional It\^{o} formula  \cite{Prato}, we can also obtain the energy conservation law for system (\ref{stochastic maxwell equations}). This result shows that the electromagnetic energy is still invariant under the multiplicative noise. This is stated in the following Theorem.

\begin{theorem}\label{them1}
Let $(\mathbf{E}, \mathbf{H})^T$ be the solution of \eqref{stochastic maxwell equations} under PEC boundary condition \eqref{PEC}.
Then for any $t\in[0, T]$,
\begin{equation}\label{6}
\begin{split}
{\Upsilon}(t)&=\int_{\Theta}(|\mathbf{E}(x,y,z,t)|^{2}+|\mathbf{H}(x,y,z,t)|^{2})d\Theta\\
&=\int_{\Theta}(|\mathbf{E}(x,y,z,t_{0})|^{2}+|\mathbf{H}(x,y,z,t_{0})|^{2})d\Theta\\
&={\Upsilon}(t_{0}),~~a.s.
\end{split}
\end{equation}

\end{theorem}
\begin{proof}
In order to prove the result \eqref{6}, we will use an equivalent It\^{o} form of \eqref{stochastic maxwell equations}. Define a function
\begin{align}\label{ito}
\Psi(x)=\sum_{m=0}^{\infty}(\sqrt{\eta_m}e_m(x))^2, ~ x\in \mathbb{R},
\end{align}
which is independent of the basis $(e_m)_{m\in\mathbb{N}}$. Then this equivalent It\^{o} equation can be rewritten as
 \begin{equation*}
\begin{split}
d\textbf{E}&=(\nabla\times\textbf{H}-\frac 1 2\lambda^2 \Psi \textbf{E})dt-\lambda\textbf{H}dW,\\
d\textbf{H}&=(-\nabla\times\textbf{E}-\frac 1 2\lambda^2 \Psi \textbf{H})dt+\lambda\textbf{E}dW.
\end{split}
\end{equation*}
Introducing the following functionals
\begin{align*}\nonumber
F_1(\textbf{E})=\int_{\Theta}|\textbf{E}(x,y,z,t)|^{2}d\Theta,\\\nonumber
F_2(\textbf{H})=\int_{\Theta}|\textbf{H}(x,y,z,t)|^{2}d\Theta.
\end{align*}
It is easy to verify that $\textbf{E}$ and $\textbf{H}$ satisfy the following first and second Fr\'{e}chet derivative
 \begin{equation}\label{3.3}
 \begin{split}
 DF_1(\textbf{E})(\varphi)=2\int_{\Theta}\langle\textbf{E},\varphi\rangle d\Theta,~D^2F_1(\textbf{E})(\varphi,\psi)&=2\int_{\Theta}\langle\psi,\varphi\rangle d\Theta,\\
 DF_2(\textbf{H})(\varphi)=2\int_{\Theta}\langle\textbf{H},\varphi\rangle d\Theta,~D^2F_2(\textbf{H})(\varphi,\psi)&=2\int_{\Theta}\langle\psi,\varphi\rangle d\Theta,
\end{split}
 \end{equation}
 where $\varphi,\psi \in L^{2}(\Theta)^3$, and $\langle\cdot, \cdot\rangle$ denotes the Euclid inner
product.

 By using the infinite dimensional It\^{o} formula for $F_1(\textbf{E}(t))$ and $F_2(\textbf{H}(t))$, we can obtain
 \begin{equation}\label{3.4}
 \begin{split}
 F_1(\textbf{E}(t))=&F_1(\textbf{E}(0))+\int_{0}^{t}\langle DF_1\textbf{E}(s),-\lambda \textbf{H}(s)dW(s)\rangle \\
 &+\int_{0}^{t}\langle DF_1(\textbf{E}(s)),\nabla\times \textbf{H}(s)-\frac 1 2\lambda^2\Psi\textbf{E}(s)\rangle ds\\
 &+\frac{\lambda^{2}}{2}\int_{0}^{t}Tr[D^2F_1(\textbf{E}(s))(\textbf{H}(s)Q^{\frac1 2})(\textbf{H}(s)Q^{\frac1 2})^*]ds.
 \end{split}
 \end{equation}
 Substituting (\ref{3.3}) into (\ref{3.4}), we have
 \begin{equation}\label{3.5}
 \begin{split}
 F_1(\textbf{E}(t))=&F_1(\textbf{E}(0))+\underbrace{2\int_{\Theta}\int_{0}^{t}\langle\textbf{E}(s),-\lambda \textbf{H}(s)dW(s)\rangle d\Theta}_{A}\\
 &+2\int_{\Theta}\int_{0}^{t}\{\langle\textbf{E}(s),\nabla\times \textbf{H}(s)-\frac 1 2\lambda^2\Psi\textbf{E}(s)\rangle\\
  &+\frac 1 2\lambda^{2}Tr[\langle\textbf{H}(s),\textbf{H}(s)\rangle Q^{\frac1 2}(Q^{\frac1 2})^*]\}dsd\Theta.
 \end{split}
 \end{equation}
Similarly, we apply It\^{o} formula to $F_2(\textbf{H}(t))$ and obtain
  \begin{equation}\label{3.6}
  \begin{split}
 F_2(\textbf{H}(t))=&F_2(\textbf{H}(0))+\underbrace{2\int_{\Theta}\int_{0}^{t}\langle\textbf{H}(s),\lambda \textbf{E}(s)dW(s)\rangle d\Theta}_{B}\\
 &-2\int_{\Theta}\int_{0}^{t}\langle\textbf{H}(s),\nabla\times \textbf{E}(s)+\frac 1 2\lambda^2\Psi\textbf{H}(s)\rangle dsd\Theta\\
 &+\lambda^{2}\int_{\Theta}\int_{0}^{t}Tr[\langle\textbf{E}(s),\textbf{E}(s)\rangle Q^{\frac1 2}(Q^{\frac1 2})^*]dsd\Theta.
 \end{split}
 \end{equation}
 Summing (\ref{3.5}) and (\ref{3.6}), we find that $A$ and $B$ extinguish, then we have
 \begin{equation*}
 \begin{split}
\int_{\Theta}(|\textbf{E}(t)|^{2}+|\textbf{H}(t)|^{2})d\Theta&=\int_{\Theta}(|\textbf{E}(0)|^{2}+|\textbf{H}(0)|^{2})d\Theta\\
&+2\underbrace{\int_{\Theta}\int_{0}^{t}\langle\textbf{E}(s),\nabla\times \textbf{H}(s)\rangle-\langle\textbf{H}(s),\nabla\times \textbf{E}(s)\rangle dsd\Theta}_{C}\\
&+\underbrace{\int_{\Theta}\int_{0}^{t}\langle\textbf{E}(s),-\lambda^2\Psi\textbf{E}(s)\rangle+ \langle\textbf{H}(s),-\lambda^2\Psi\textbf{H}(s)\rangle dsd\Theta}_{D}\\
&+\underbrace{\int_{\Theta}\int_{0}^{t}\lambda^2\langle\textbf{H}(s), \textbf{H}(s)\rangle Tr(Q)+\lambda^2\langle\textbf{E}(s), \textbf{E}(s)\rangle Tr(Q)dsd\Theta}_{P}.
\end{split}
\end{equation*}
Here $Tr(Q)$ denotes the trace of operate $Q$, i.e., $Tr(Q) =\sum\limits_{m\in \mathbb{N}}\langle Qe_m, e_m\rangle_{L_2}=\sum\limits_{m\in \mathbb{N}}\eta_m.$
By Green's formula and PEC boundary condition (\ref{PEC}), the term $C$ satisfies the following equality
\begin{equation*}\label{A}
\begin{split}
C&=-\int_{0}^{t}\int_{\Theta}\nabla\cdot(\textbf{E}\times\textbf{H})d\Theta ds\\
&=-\int_{0}^{t}\int_{\partial\Theta}(\textbf{E}\times\textbf{H})\cdot \textbf{n}d\Theta ds\\
&=0.
\end{split}
\end{equation*}
It follows from the definition of $\Psi$ that
$D+P=0.$
\end{proof}

\begin{remark}
 Equality (\ref{6}) is an important criterion in constructing efficient numerical methods for computing the propagation of electromagnetic wave and in measuring whether a numerical simulation method is good or not.
\end{remark}

\section{Energy-conserving method}
In this section, we propose an energy-conserving numerical method for \eqref{stochastic maxwell equations}.
It is a combination of midpoint method in time and wavelet collocation method in space.

\subsection{Discretization in time and space}
Applying the midpoint method to equation \eqref{stochastic maxwell equations} in temporal direction, we get
\begin{equation}\label{42}
\begin{split}
H_{1}^{n+1}=H_{1}^{n}+\Delta t\Big(\frac{\partial }{\partial
z}E_{2}^{n+1/2}-\frac{\partial}{\partial
y}E_{3}^{n+1/2}\Big)+\lambda E_{1}^{n+1/2}\Delta W^{n},\\
H_{2}^{n+1}=H_{2}^{n}+\Delta t\Big(\frac{\partial}{\partial x}E_{3}^{n+1/2}-\frac{\partial }{\partial z}E_{1}^{n+1/2}\Big)+\lambda E_{2}^{n+1/2}\Delta W^{n},\\
H_{3}^{n+1}=H_{3}^{n}+\Delta t\Big(\frac{\partial}{\partial
y}E_{1}^{n+1/2} -\frac{\partial}{\partial
x}E_{2}^{n+1/2}\Big)+\lambda E_{3}^{n+1/2}\Delta W^{n},\\
E_{1}^{n+1}=E_{1}^{n}-\Delta t\Big(\frac{\partial }{\partial z}H_{2}^{n+1/2}-\frac{\partial}{\partial
y}H_{3}^{n+1/2}\Big)-\lambda H_{1}^{n+1/2}\Delta W^{n},\\
E_{2}^{n+1}=E_{2}^{n}-\Delta t\Big(\frac{\partial }{\partial x}H_{3}^{n+1/2}-\frac{\partial}{\partial
z}H_{1}^{n+1/2}\Big)-\lambda H_{2}^{n+1/2}\Delta W^{n},\\
E_{3}^{n+1}=E_{3}^{n}-\Delta t\Big(\frac{\partial }{\partial y}H_{1}^{n+1/2}-\frac{\partial}{\partial
x}H_{2}^{n+1/2}\Big)-\lambda H_{3}^{n+1/2}\Delta W^{n},
\end{split}
\end{equation}
where $\Delta t$ is the temporal step-size and $u^{n+1/2}=\frac{1}{2}(u^{n}+u^{n+1})$. In the sequel,
\begin{align}\label{MMM}
\Delta W^{n}=W(t_{n+1})-W(t_{n})=\sum_{m=1}^{\mathcal{M}}\sqrt{\eta_m}(\beta_m(t_{n+1})-\beta_{m}(t_{n}))e_m,
\end{align}
where $\mathcal{M}$ is a positive integer.
Then, we apply wavelet collocation method to discretize (\ref{42}) in spatial direction and obtain the full-discrete stochastic multi-symplectic wavelet collocation method. Now we give some preliminary results of wavelet collocation method. For more details, see \cite{Bertoluzza} and references therein.

A Daubechies scaling function $\phi(x)$ of order $\gamma$ satisfies
\begin{align*}
\phi(x)=\sum_{k=0}^{\gamma-1}h_k\phi(2x-k),
\end{align*}
 where $\gamma$ is a positive even integer and $\{h_k\}_{k=0}^{\gamma-1}$ are $\gamma$ non-vanishing ``filter coefficients". Define the autocorrelation function $\theta(x)$ of $\phi(x)$ as
 \begin{align*}
 \theta(x)=\int\phi(x)\phi(t-x)dt.
 \end{align*}
 Suppose that $V_j$ is the linear span of $\{\theta_{jk}(x)=2^{j/2}\theta(2^jx-k),k\in \mathbb{Z}\}$. It can be proved that $(V_j)_{j\in \mathbb{Z}}$ forms
 a multiresolution analysis.

Consider $E_{1}(x,y,z,t)$ defined on spatial domain
$[0,L_{1}]\times[0,L_{2}]\times[0,L_{3}]$ with
$N_{1}\times N_{2}\times N_{3}$ grid points, where
$N_{1}=L_{1}\cdot 2^{J_{1}}$, $N_{2}=L_{2}\cdot 2^{J_{2}}$,
$N_{3}=L_{3}\cdot 2^{J_{3}}$.
Interpolating it at
collocation points $(x_{i},y_{j},z_{k})=(i/2^{J_{1}},j/2^{J_{2}},k/2^{J_{3}})$,
$i=1,...,N_{1}$, $j=1,...,N_{2}$, $k=1,...,N_{3}$ gives
\begin{equation}\label{52}
IE_{1}(x,y,z,t)=\sum_{i=1}^{N_{1}}\sum_{j=1}^{N_{2}}\sum_{k=1}^{N_{3}}E_{1_{i,j,k}}\theta(2^{J_{1}}x-i)\theta(2^{J_{2}}y-j)\theta(2^{J_{3}}z-k).
\end{equation}
Making partial differential with respect to $y$ and
evaluating the resulting expression at collocation points, we
obtain
\begin{equation}\label{53}
\begin{split}
\frac{\partial IE_{1}(x_{i},y_{j},z_{k},t)}{\partial
y}&=\sum_{i^{'}=1}^{N_{1}}\sum_{j^{'}=1}^{N_{2}}\sum_{k^{'}=1}^{N_{3}}E_{1_{i^{'},j^{'},k^{'}}}\theta(2^{J_{1}}x_{i}-i^{'})
\theta(2^{J_{3}}z_{k}-k^{'})\frac{d\theta(2^{J_{2}}y-j^{'})}{dy}|_{y_{j}}\nonumber\\
&=\sum_{j^{'}=1}^{N_{2}}E_{1_{i,j^{'},k}}(2^{J_{2}}\theta^{'}(j-j^{'}))=\sum_{j^{'}=j-(\gamma-1)}^{j+(\gamma-1)}E_{1_{i,j^{'},k}}(B^{y})_{j,j^{'}}\\
&=((I_{N_{1}}\otimes B^{y}\otimes
I_{N_{3}})\textbf{E}_{1})_{i,j,k},
\end{split}
\end{equation}
where $\otimes$ means Kronecker inner product and $I_{N_{1}}$ is the
$N_{1}\times N_{1}$ identity matrix. $\mathbf{E_{1}}=((E_{1})_{1,1,1},(E_{1})_{2,1,1},(E_{1})_{N_{1},1,1},\cdots,(E_{1})_{1,N_{2},1},\cdots,(E_{1})_{N_{1},N_{2},N_{3}})^{T}$. The differential matrix $B^{y}$ for the first-order partial differential operator $\partial_y$ is an $N_{2}\times
N_{2}$ sparse skew-symmetric circulant matrix with entries
\begin{equation*}
(B^y)_{m,m'}=\left\{\begin{array}{ccccccc}
2^{J_2}\theta'(m-m'),~~m-(\gamma-1)\leq m'\leq m+(\gamma-1);\\[3mm]
2^{J_2}\theta'(-l),~~~~~~~~~m-m'=N_{2}-l,~~~1\leq l\leq \gamma-1;\\[3mm]
2^{J_2}\theta'(l),~~~~~~~~~~~~m'-m=N_{2}-l,~~~1\leq l\leq \gamma-1;\\[3mm]
0,~~~~~~~~~~~~~~~~~~~~~~~~~~~~~~~~~~~~~~~\textrm{otherwise}.
\end{array}\right.
\end{equation*}
Using the similar manner, we can obtain the discrete differential matrices $B^{x}\otimes I_{N_{2}}\otimes
I_{N_{3}}$ and $I_{N_{1}}\otimes
I_{N_{2}}\otimes B^{z}$ corresponding to $\partial_x$ and $\partial_z$, respectively. Now we have a
full-discrete method for stochastic Maxwell equations as
\begin{equation}\label{54}
\begin{split}
(\mathbf{E_{1}})^{n+1}-(\mathbf{E_{1}})^{n}&=\Delta t\Big(A_{2}(\mathbf{H_{3}})^{n+1/2}-A_{3}(\mathbf{H_{2}})^{n+1/2}\Big)-\lambda(\mathbf{H_{1}})^{n+1/2} \mathbf{W}^{n},\\[3mm]
(\mathbf{E_{2}})^{n+1}-(\mathbf{E_{2}})^{n}&=\Delta t\Big(A_{3}(\mathbf{H_{1}})^{n+1/2}-A_{1}(\mathbf{H_{3}})^{n+1/2}\Big)-\lambda(\mathbf{H_{2}})^{n+1/2} \mathbf{W}^{n},\\[3mm]
(\mathbf{E_{3}})^{n+1}-(\mathbf{E_{3}})^{n}&=\Delta t\Big(A_{1}(\mathbf{H_{2}})^{n+1/2}-A_{2}(\mathbf{H_{1}})^{n+1/2}\Big)-\lambda(\mathbf{H_{3}})^{n+1/2} \mathbf{W}^{n},\\[3mm]
(\mathbf{H_{1}})^{n+1}-(\mathbf{H_{1}})^{n}&=\Delta t\Big(A_{3}(\mathbf{E_{2}})^{n+1/2}-A_{2}(\mathbf{E_{3}})^{n+1/2}\Big)+\lambda(\mathbf{E_{1}})^{n+1/2} \mathbf{W}^{n},\\[3mm]
(\mathbf{H_{2}})^{n+1}-(\mathbf{H_{2}})^{n}&=\Delta t\Big(A_{1}(\mathbf{E_{3}})^{n+1/2}-A_{3}(\mathbf{E_{1}})^{n+1/2}\Big)+\lambda(\mathbf{E_{2}})^{n+1/2} \mathbf{W}^{n},\\[3mm]
(\mathbf{H_{3}})^{n+1}-(\mathbf{H_{3}})^{n}&=\Delta t\Big(A_{2}(\mathbf{E_{1}})^{n+1/2}-A_{1}(\mathbf{E_{2}})^{n+1/2}\Big)+\lambda(\mathbf{E_{3}})^{n+1/2} \mathbf{W}^{n},
\end{split}
\end{equation}
where $A_{1}=B^{x}\otimes I_{N_{2}}\otimes I_{N_{3}}$, $A_{2}=I_{N_{1}}\otimes B^{y}\otimes I_{N_{3}}$ and
$A_{3}=I_{N_{1}}\otimes I_{N_{2}}\otimes B^{z}$ are skew-symmetric, $\mathbf{W}^{n}=\mathbf{e}\otimes (\Delta W_{1}^{n},\cdots,\Delta W_{N_{1}}^{n})^{T}$, $\mathbf{e}=(1,1,\cdots,1,1)_{N_{2}\times N_{3}}^{T}$, $\Delta W_{i}^{n}$ is an approximation of $\Delta W^n$ in spatial direction with respect to $x$. And $(\mathbf{H_{1}})^{n+1/2}\mathbf{W}^{n}$ denotes the components multiplication between $(\mathbf{H_{1}})^{n+1/2}$ and $\mathbf{W}^{n}$, respectively.

\subsection{Properties of the method}
In this section, we show that the method (\ref{54}) preserves the discrete stochastic multi-symplectic conservation law and discrete energy conservation law.
\subsubsection{Stochastic multi-symplecticity}
The discrete stochastic multi-symplectic conservation law is stated as follows.
\begin{theorem}
The method \eqref{54} has
the following discrete stochastic multi-symplectic conservation law
\begin{align}\label{55}
&\frac{\omega^{n+1}_{i,j,k}-\omega^{n}_{i,j,k}}{\Delta t}+\sum^{i+(\gamma-1)}_{i^{'}=i-(\gamma-1)}(B^{x})_{i,i^{'}}(\kappa_{x})^{n+1/2}_{i^{'},j,k}\\[5mm]\nonumber
&+\sum^{j+(\gamma-1)}_{j^{'}=j-(\gamma-1)}(B^{y})_{j,j^{'}}(\kappa_{y})^{n+1/2}_{i,j^{'},k}
+\sum^{k+(\gamma-1)}_{k^{'}=k-(\gamma-1)}(B^{z})_{k,k^{'}}(\kappa_{z})^{n+1/2}_{i,j,k^{'}}=0,
\end{align}
where
 \begin{align*}
\omega^{n}_{i,j,k}=\frac{1}{2}d\textbf{u}^{n}_{i,j,k}\wedge M
d\textbf{u}^{n}_{i,j,k},~~(\kappa_{x})^{n+1/2}_{i^{'},j,k}=
d\textbf{u}^{n+1/2}_{i,j,k}\wedge
K_{1}d\textbf{u}^{n+1/2}_{i^{'},j,k},\\[3mm](\kappa_{y})^{n+1/2}_{i,j^{'},k}=d\textbf{u}^{n+1/2}_{i,j,k}\wedge
K_{2} d\textbf{u}^{n+1/2}_{i,j^{'},k},~~(\kappa_{z})^{n+1/2}_{i,j,k^{'}}=d\textbf{u}^{n+1/2}_{i,j,k}\wedge
K_{3} d\textbf{u}^{n+1/2}_{i,j,k^{'}}.
\end{align*}
\end{theorem}
\begin{proof}
Let $\textbf{u}=(\mathbf{H_{1}},\mathbf{H_{2}},\mathbf{H_{3}},\mathbf{E_{1}},\mathbf{E_{2}},\mathbf{E_{3}})^{T}$. (\ref{54}) can be rewritten as
\begin{eqnarray*}
M\frac{\textbf{u}^{n+1}_{i,j,k}-\textbf{u}^{n}_{i,j,k}}{\Delta t}&&+\sum^{i+(\gamma-1)}_{i^{'}=i-(\gamma-1)}(B^{x})_{i,i^{'}}(K_{1}\textbf{u}^{n+1/2}_{i^{'},j,k})
+\sum^{j+(\gamma-1)}_{j^{'}=j-(\gamma-1)}(B^{y})_{j,j^{'}}(K_{2}\textbf{u}^{n+1/2}_{i,j^{'},k})\\[5mm]\nonumber
&&+\sum^{k+(\gamma-1)}_{k^{'}=k-(\gamma-1)}(B^{z})_{k,k^{'}}(K_{3}\textbf{u}^{n+1/2}_{i,j,k^{'}})=\nabla_{\textbf{u}}S(\textbf{u}^{n+1/2}_{i,j,k})\frac{\Delta W_{i}^{n}}{\Delta t}.
\end{eqnarray*}
The variational form associated with the above equation is
\begin{align}\label{57}
M\frac{d\textbf{u}^{n+1}_{i,j,k}-d\textbf{u}^{n}_{i,j,k}}{\Delta t}&+\sum^{i+(\gamma-1)}_{i^{'}=i-(\gamma-1)}(B^{x})_{i,i^{'}}(K_{1}d\textbf{u}^{n+1/2}_{i^{'},j,k})+
\sum^{j+(\gamma-1)}_{j^{'}=j-(\gamma-1)}(B^{y})_{j,j^{'}}(K_{2}d\textbf{u}^{n+1/2}_{i,j^{'},k})\\[5mm]\nonumber
&+\sum^{k+(\gamma-1)}_{k^{'}=k-(\gamma-1)}(B^{z})_{k,k^{'}}(K_{3}d\textbf{u}^{n+1/2}_{i,j,k^{'}})
=\nabla^{2}S(\textbf{u}^{n+1/2}_{i,j,k})d\textbf{u}^{n+1/2}_{i,j,k}\frac{\Delta W_{i}^{n}}{\Delta t}.\nonumber
\end{align}
Taking the wedge product with $d\textbf{u}^{n+1/2}_{i,j,k}$ on both
sides of (\ref{57}), we finish the proof.
\end{proof}
\subsubsection{Energy preservation}
In this subsection, we will state the discrete energy conservation law.
\begin{theorem}\label{Theorem}
Under periodic boundary conditions, the stochastic multi-symplectic
wavelet collocation method \eqref{54} has the following discrete energy
conservation law
\begin{align}\label{58}
\|\mathbf{E}^{n+1}\|^{2}+\|\mathbf{H}^{n+1}\|^{2}=\|\mathbf{E}^{n}\|^{2}+\|\mathbf{H}^{n}\|^{2},~a.s.,
\end{align}
where
\begin{eqnarray}
\|{\bf E}^{n}\|^{2}=\Delta x \Delta y\Delta z\sum^{N_{1}}_{i=1}\sum^{N_{2}}_{j=1}\sum^{N_{3}}_{k=1}\Big((E^{n}_{1_{i,j,k}})^{2}+(E^{n}_{2_{i,j,k}})^{2}+(E^{n}_{3_{i,j,k}})^{2}\Big),\nonumber\\
\|{\bf H}^{n}\|^{2}=\Delta x \Delta y\Delta z\sum^{N_{1}}_{i=1}\sum^{N_{2}}_{j=1}\sum^{N_{3}}_{k=1}\Big((H^{n}_{1_{i,j,k}})^{2}+(H^{n}_{2_{i,j,k}})^{2}+(H^{n}_{3_{i,j,k}})^{2}\Big).\nonumber
\end{eqnarray}
\end{theorem}
\begin{proof}
We make inner product
of (\ref{54}) with $\mathbf{E}_{1}^{n+1/2}$, $\mathbf{E}_{2}^{n+1/2}$,
$\mathbf{E}_{3}^{n+1/2}$, $\mathbf{H}_{1}^{n+1/2}$, $\mathbf{H}_{2}^{n+1/2}$, $\mathbf{H}_{3}^{n+1/2}$,
respectively, it yields
\begin{equation*}
\begin{split}
(\mathbf{E}_{1}^{n+1})^{2}
-(\mathbf{E}_{1}^{n})^{2}=2(A_{2}\mathbf{H}_{3}^{n+1/2}-A_{3}\mathbf{H}_{2}^{n+1/2})\mathbf{E}_{1}^{n+1/2}\Delta t-2\lambda \mathbf{H}_{1}^{n+1/2}\mathbf{E}_{1}^{n+1/2}\mathbf{W}^{n},\\[3mm]
(\mathbf{E}_{2}^{n+1})^{2}-(\mathbf{E}_{2}^{n})^{2}
=2(A_{3}\mathbf{H}_{1}^{n+1/2}-A_{1}\mathbf{H}_{3}^{n+1/2})\mathbf{E}_{2}^{n+1/2}\Delta t-2\lambda \mathbf{H}_{2}^{n+1/2}\mathbf{E}_{2}^{n+1/2} \mathbf{W}^{n},\\[3mm]
(\mathbf{E}_{3}^{n+1})^{2}-(\mathbf{E}_{3}^{n})^{2}
=2(A_{1}\mathbf{H}_{2}^{n+1/2}-A_{2}\mathbf{H}_{1}^{n+1/2})\mathbf{E}_{3}^{n+1/2}\Delta t-2\lambda \mathbf{H}_{3}^{n+1/2}\mathbf{E}_{3}^{n+1/2}\mathbf{W}^{n},\\[3mm]
(\mathbf{H}_{1}^{n+1})^{2}-(\mathbf{H}_{1}^{n})^{2}
=2(A_{3}\mathbf{E}_{2}^{n+1/2}-A_{2}\mathbf{E}_{3}^{n+1/2})\mathbf{H}_{1}^{n+1/2}\Delta t+2\lambda \mathbf{E}_{1}^{n+1/2}\mathbf{H}_{1}^{n+1/2}\mathbf{W}^{n},\\[3mm]
(\mathbf{H}_{2}^{n+1})^{2}-(\mathbf{H}_{2}^{n})^{2}
=2(A_{1}\mathbf{E}_{3}^{n+1/2}-A_{3}\mathbf{E}_{1}^{n+1/2})\mathbf{H}_{2}^{n+1/2}\Delta t+2\lambda \mathbf{E}_{2}^{n+1/2}\mathbf{H}_{2}^{n+1/2}\mathbf{W}^{n},\\[3mm]
(\mathbf{H}_{3}^{n+1})^{2}-(\mathbf{H}_{3}^{n})^{2}
=2(A_{2}\mathbf{E}_{1}^{n+1/2}-A_{1}\mathbf{E}_{2}^{n+1/2})\mathbf{H}_{3}^{n+1/2}\Delta t+2\lambda
\mathbf{E}_{3}^{n+1/2}\mathbf{H}_{3}^{n+1/2}\mathbf{W}^{n}.
\end{split}
\end{equation*}
Summing over subscripts $i,j,k$, exploiting the skew-symmetric property of $A_{i}$ and using the periodic boundary conditions reads \eqref{58}.
\end{proof}

The result of this theorem is evidently consistent with (\ref{6}), which means that the energy can be preserved by the proposed stochastic multi-symplectic wavelet collocation method. In the next section, we will verify this conservation law numerically.

\section{Numerical experiments}
This section will provide numerical experiments to test the new derived stochastic multi-symplectic wavelet collocation method (\ref{54}). We focus on the preservation of energy, the mean square convergence order in time and a comparison with a finite difference method.

When we take no account of the noise term, i.e., $\lambda=0$, \eqref{stochastic maxwell equations} reduces to 3D deterministic Maxwell equations.
In our numerical calculations, initial values
\begin{equation}\label{111}
\begin{split}
E_{1_{0}}&=\cos(2\pi(x+y+z)),~E_{2_{0}}=-2E_{1_0},~E_{3_0}=E_{1_0},\\
H_{1_0}&=\sqrt{3}E_{1_0},~H_{2_0}=0,~H_{3_0}=-\sqrt{3}E_{1_0}.
\end{split}
\end{equation}
 on $\Theta=[0,1]^3$ and periodic boundary conditions are considered.
In the following experiments, we take $\Delta t=0.005$ and
 $\Delta x=\Delta y=\Delta z=1/2^{5}$. We use the order of the Daubechies scaling function $\gamma=10$ to solve the problem till time $T=20$ and $T=200$, respectively.
In the sequel, we take the orthonormal basis $(e_m)_{m\in\mathbb{N}}$ and eigenvalue $(\eta_{m})_{m\in\mathbb{N}}$ as
\begin{equation}\label{123}
e_{m}(x)=\sqrt{2}\sin(m\pi x),~~\eta_{m}=\frac{1}{m^{2}},
\end{equation}
then $\Delta W_{i}^{n}$ can be regarded as an approximation of integral
\begin{equation}\label{1234}
\Delta W_{i}^{n}=\frac{1}{\Delta x}\int_{i\Delta x}^{(i+1)\Delta x}\int_{t_{n}}^{t_{n+1}}\sum_{m=1}^{\mathcal{M}}\sqrt{\eta_{m}}e_{m}(x)d\beta_{m}(s)dx.
\end{equation}
Substituting (\ref{123}) into (\ref{1234}) yields
\begin{align}\label{Noise}
\Delta W_i^n=\frac{1}{\Delta x}\sum_{m=1}^{\mathcal{M}}\frac{\sqrt{2\eta_m}}{m\pi}\Big[\cos(m\pi i \Delta x)-\cos(m\pi(i+1) \Delta x)\Big]\Big[\beta_m(t_{n+1})-\beta_m(t_{n})\Big],
\end{align}
where $(\beta_m(t_{n+1})-\beta_m(t_{n}))/\sqrt{\Delta t}$ is a sequence of independent and $\mathcal{N}(0,1)$ distributed random variables.
In the sequel, we choose $\mathcal{M}=200$ in \eqref{Noise}.
\begin{figure}[th!]
\begin{center}
  \includegraphics[width=0.8\textwidth]{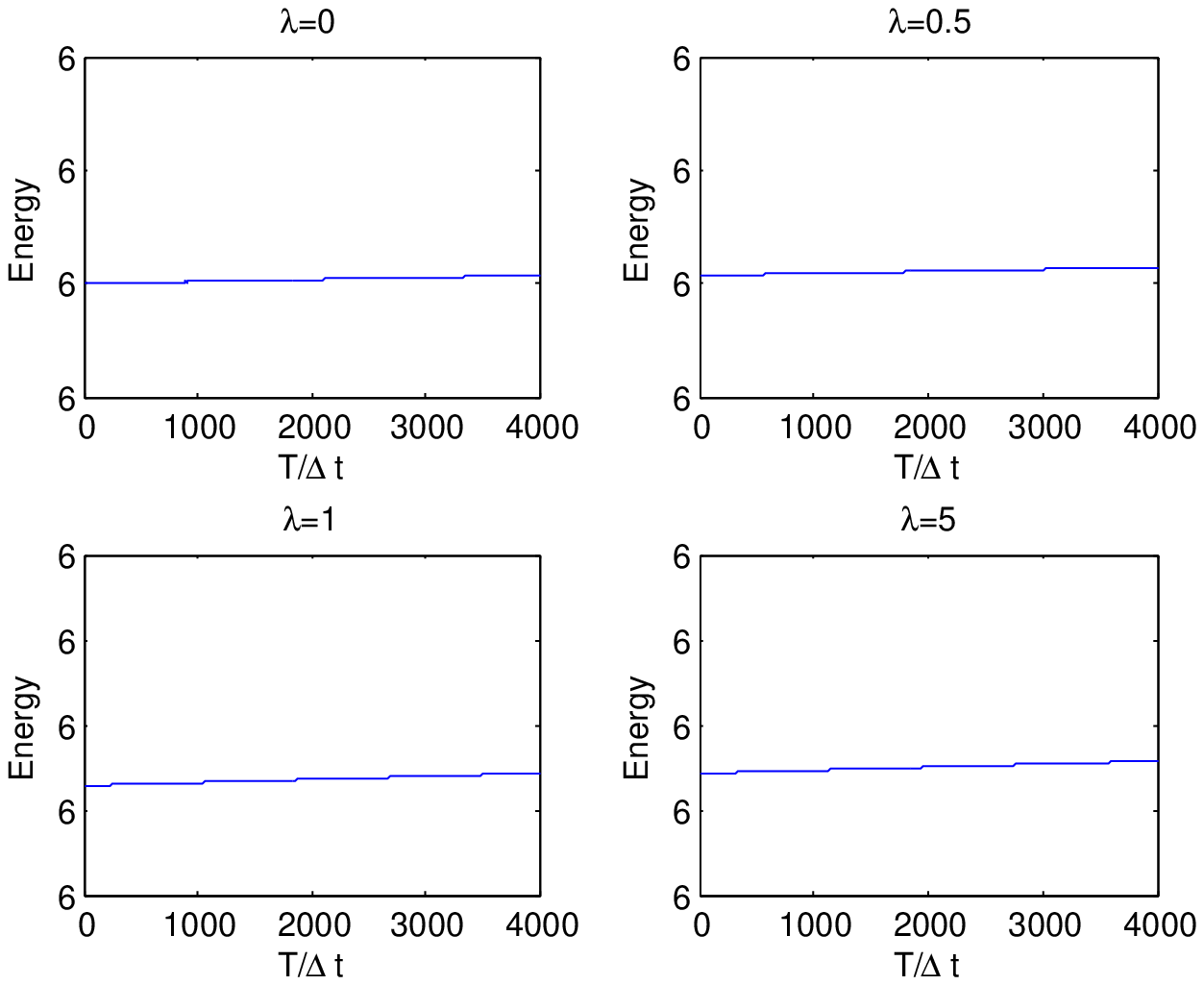}
  \caption{Evolution of the energy over one trajectory until $T=20$ with $\Delta t=0.005$ as $\lambda=0$, $\lambda=0.5$, $\lambda=1$ and $\lambda=5$, respectively.}\label{5.8}
  \end{center}
\end{figure}

\begin{figure}[th!]
\begin{center}
  \includegraphics[width=0.8\textwidth]{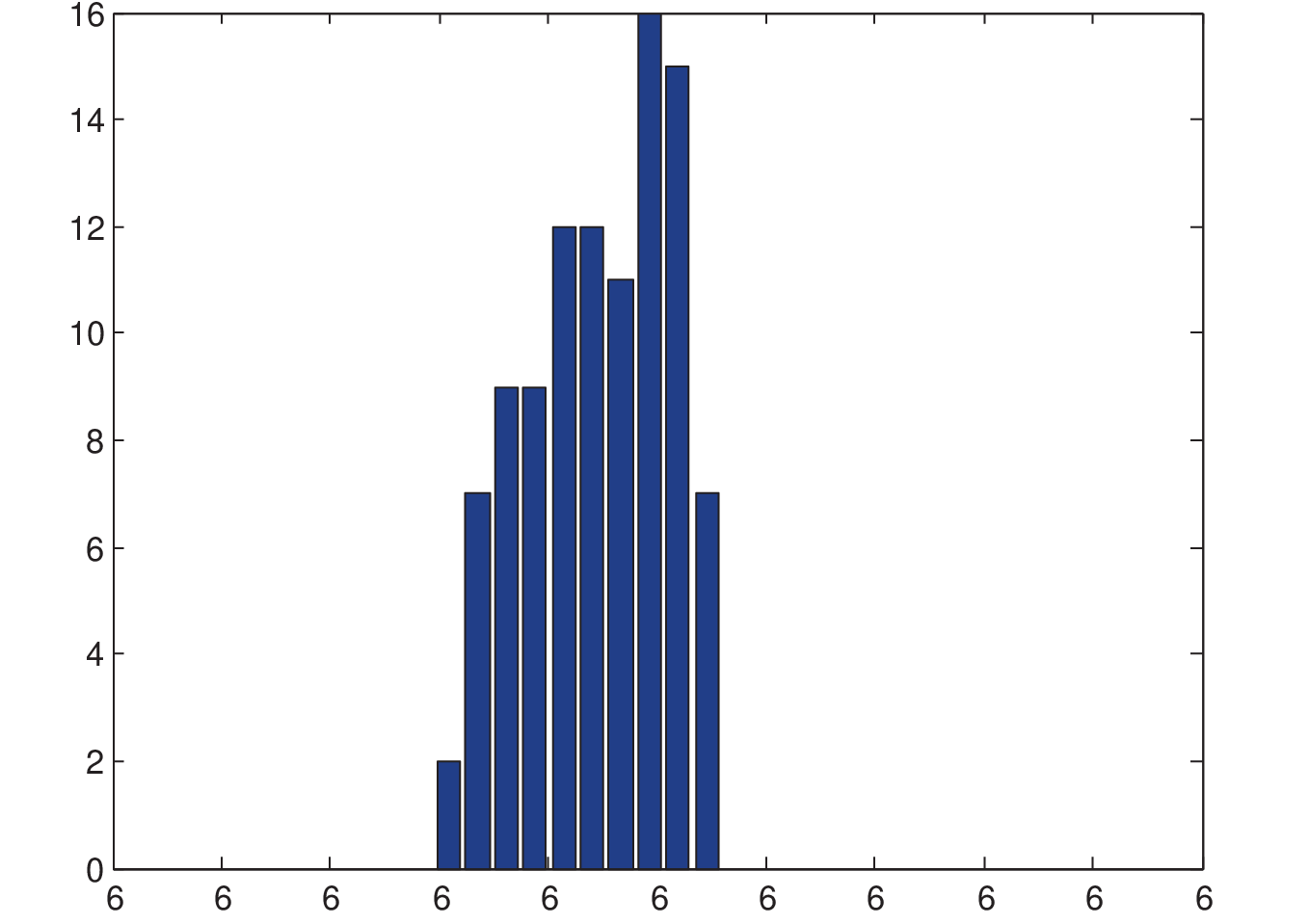}
  \caption{The probability density function of $\max\limits_n\Big(\|\mathbf{E}^n\|^2+\|\mathbf{H}^n\|^2\Big)$ in the sense of $L^2$.}\label{5.2}
  \end{center}
\end{figure}

Fig. \ref{5.8} shows the evolution of the discrete energy of one trajectory with different size of noise $\lambda=0$, $\lambda=0.5$, $\lambda=1$ and $\lambda=5$. From the figure, the energy remains to be straight horizontal lines approximately for different sizes of noise, which coincides with the theoretical analysis.
Meanwhile, from Fig. \ref{5.2} displayed the probability density function of random variable $\max\limits_n\Big(\|\mathbf{E}^n\|^2+\|\mathbf{H}^n\|^2\Big)$, we may observe that the averaged energy is bounded.
Fig. \ref{5.9} illustrates the global residuals of the discrete energy of one trajectory, i.e., $(Err)^{n}:=\Upsilon^{n}-\Upsilon^{0}$. Obviously, the residuals reach the magnitude of $10^{-11}$ for the above values of $\lambda$. All these results indicate that the proposed method preserves the discrete energy conservation law no matter how big the size of noise is. The stability in long-time computations related to the discrete energy conservation law is illustrated in Fig. \ref{5.11}. The discrete energy is almost invariant up to $T=200$ with the above sizes of noise.
The error
 in the discrete energy can be still controlled in the scale of $10^{-10}$. It can be seen that the magnitude of the global residuals become bigger than the one at $T=20$.
 \begin{figure}[th!]
\begin{center}
  \includegraphics[width=0.8\textwidth]{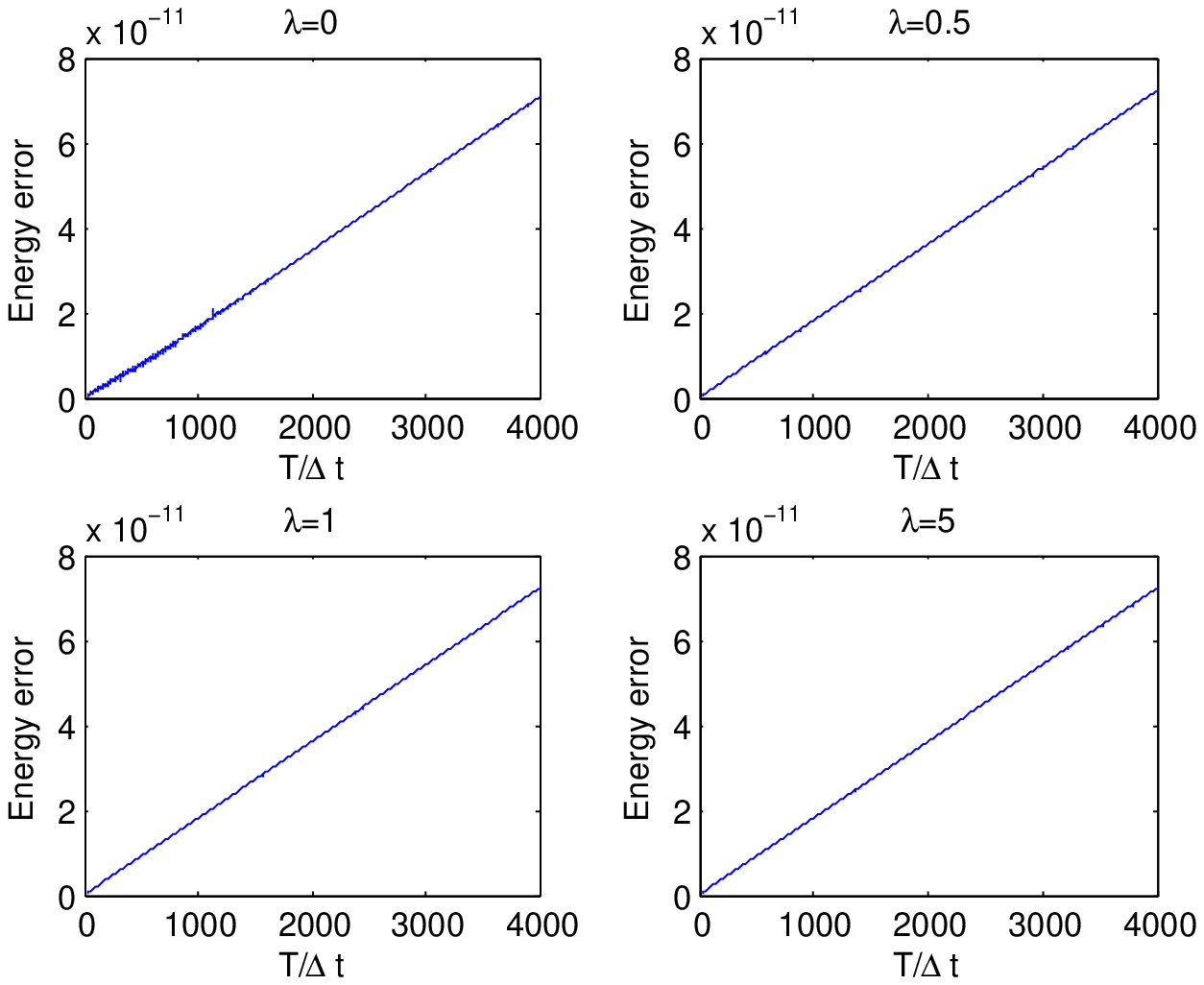}
  \caption{The global errors of discrete energy over one trajectory until $T=20$ with $\Delta t=0.005$ as $\lambda=0$, $\lambda=0.5$, $\lambda=1$ and $\lambda=5$, respectively.}\label{5.9}
  \end{center}
\end{figure}
\begin{figure}[th!]
\begin{center}
  \includegraphics[width=0.8\textwidth]{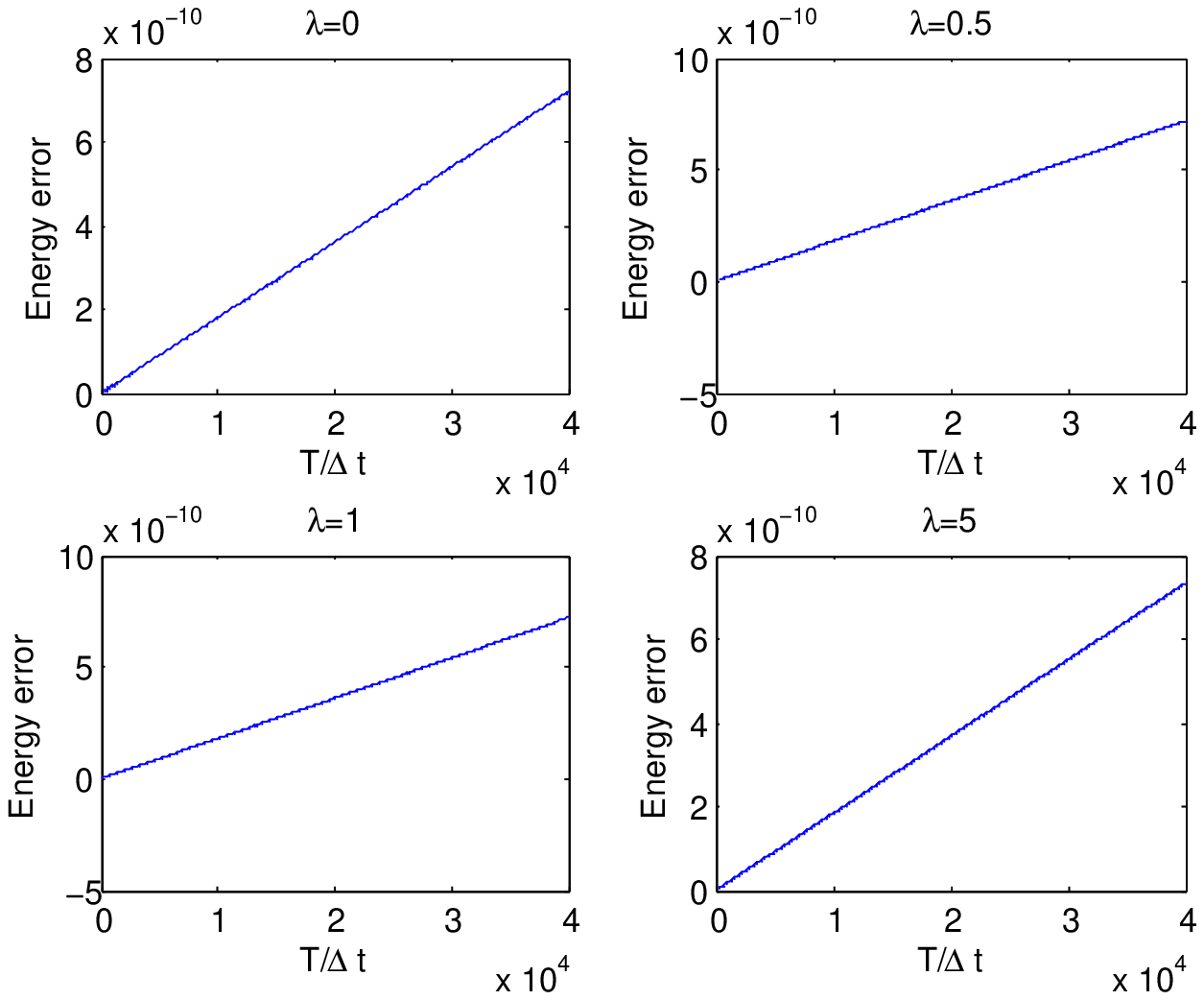}
  \caption{The global errors of discrete energy over one trajectory until $T=200$ with $\Delta t=0.005$ as $\lambda=0$, $\lambda=0.5$, $\lambda=1$ and $\lambda=5$, respectively.}\label{5.11}
  \end{center}
\end{figure}

\begin{figure}[th!]
\begin{center}
  \includegraphics[width=0.8\textwidth]{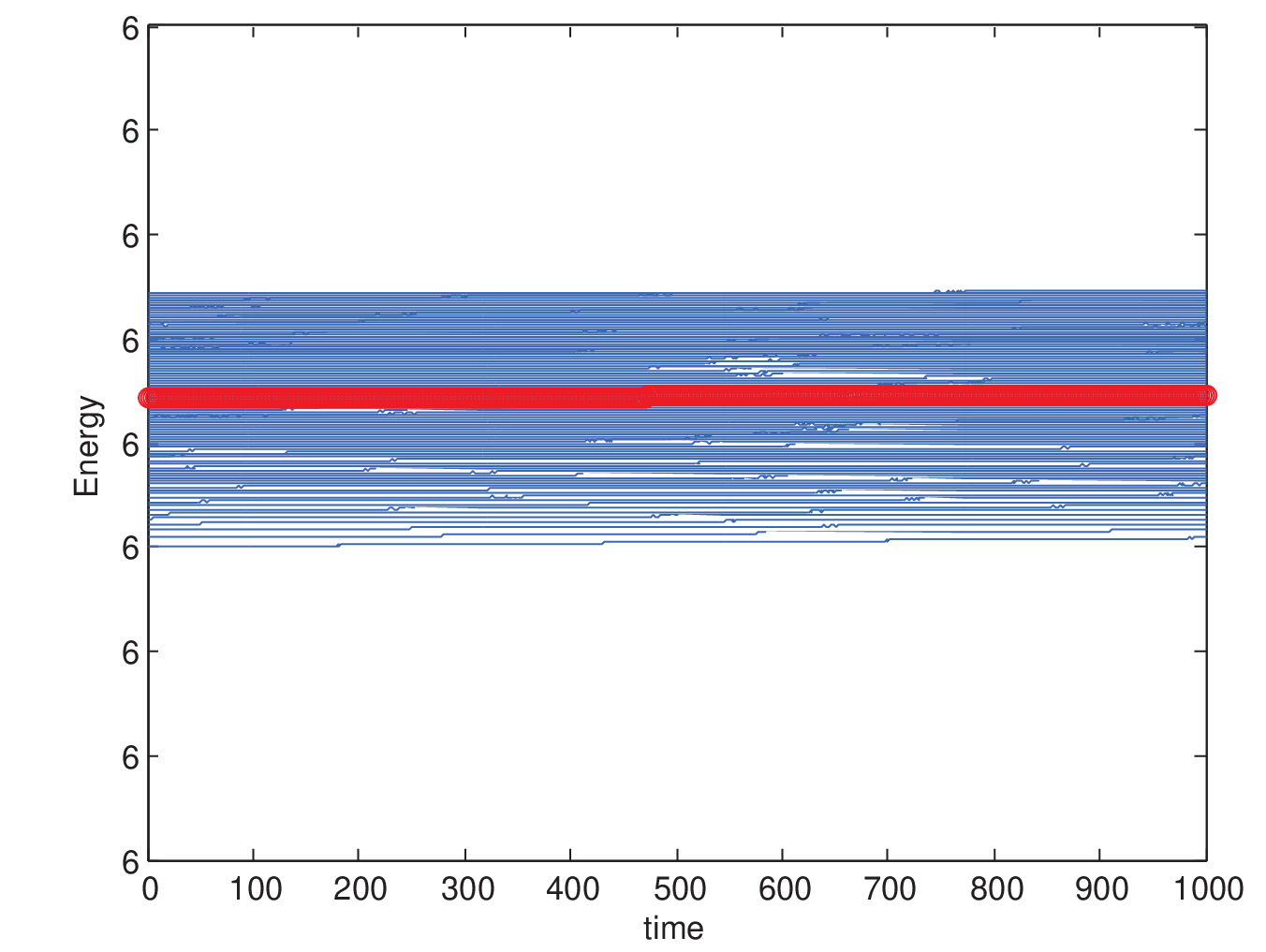}
  \caption{The averaged energy over 100 trajectories for $\lambda=0.5$.}\label{5.10}
  \end{center}
\end{figure}

In Fig. \ref{5.10}, the blue lines denote energy along 100
trajectories and the red line represents the averaged energy with $T=5$ and $\lambda=0.5$, respectively.
Obviously, the averaged energy is  nearly horizontal line with respect to time, which coincides with the continuous case.


%

In order to show the superiority of the wavelet collocation method in the preservation of energy,
we discretize the stochastic Maxwell equations \eqref{stochastic maxwell equations} by the central finite difference method.
Define the following normalized energy
\begin{equation*}
  \text{Normalized~~energy}:=(\Upsilon^{n}-\Upsilon^{0})\times 10^7,
\end{equation*}
where $\Upsilon^{n}$ and $\Upsilon^{0}$ denote the discrete energy at $t_{n}$ and $t_{0}$, respectively.
Fig. \ref{energy_com} exhibits the discrete averaged normalized energy over 100 trajectories with $\lambda=\sqrt{2}$ and $T=10$, respectively. We observe that the wavelet collocation method preserves the energy very well. However, the discrete averaged normalized energy obtained by finite difference method shows a rapid growth as time evolves. Furthermore, we plot the energy errors of the two methods in Fig. \ref{energy_error_com}. From the figure, it can be seen that the proposed method is more accurate than the finite difference method in the preservation of energy.
\begin{figure}[th!]
\begin{center}
  \includegraphics[width=0.8\textwidth]{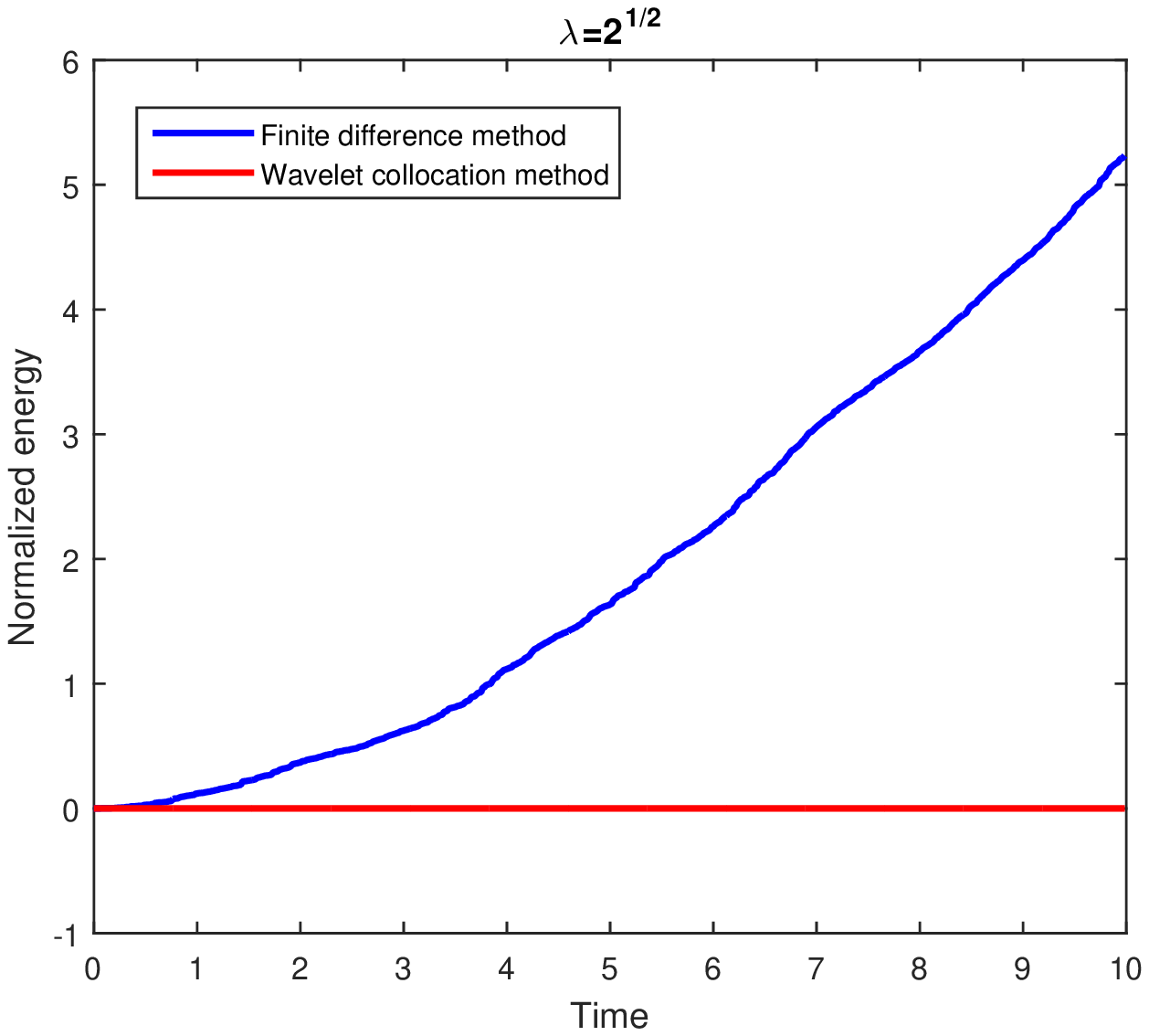}
  \caption{Evolution of the energy averaged over 100 trajectories with $\Delta t=1/64,~T=10$.}\label{energy_com}
  \end{center}
\end{figure}

\begin{figure}[th!]
\begin{center}
  \includegraphics[width=0.8\textwidth]{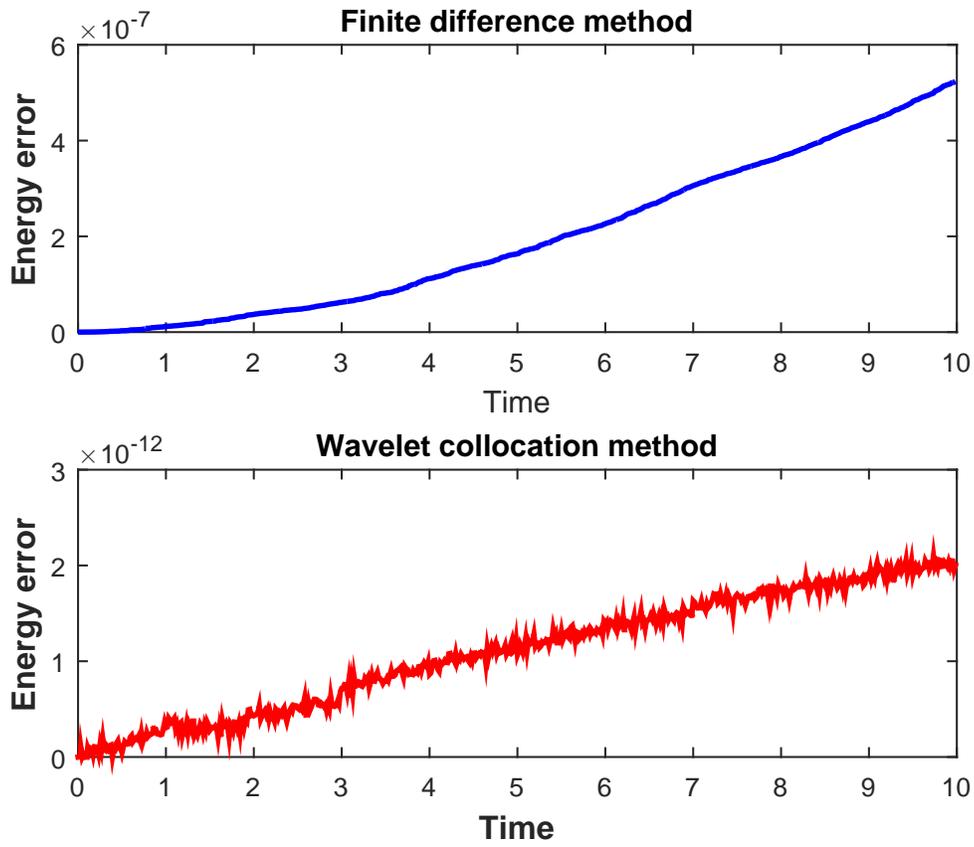}
  \caption{Evolution of the averaged energy error over 100 trajectories with $\Delta t=1/64,~T=10$ . Top: finite difference method; Bottom: wavelet collocation method.}\label{energy_error_com}
  \end{center}
\end{figure}

Finally, we investigate the convergence order in time of the proposed method.
Define
\begin{equation*}
  e_{\Delta t}^{strong}:=\Big(\mathbb{E}\|u(\cdot,T)-u_{T}(\cdot)\|^{2}\Big)^{\frac{1}{2}},
\end{equation*}
with $u=({\bf E},{\bf H})^T$. We plot
 $e_{\Delta t}^{strong}$ against $\Delta t$ on a log-log scale for the truncated number of Wiener process $1\leq \mathcal{M}\leq 8$ until $T=0.1$. We compute the numerical solution with a vary small time step, e.g., $\Delta t=2^{-11}$, as the reference ``exact" solution.
(i) We consider $\lambda=0$. The $e_{\Delta t}^{strong}$ and the convergence orders of the wavelet collocation method in time are listed in Table \ref{tab1}. We can see that the method gives a uniform second order of accuracy for deterministic Maxwell equations.
\begin{table}[htbp]
\begin{center}
\caption{Accuracy test for stochastic Maxwell equations \eqref{stochastic maxwell equations}}
 \label{tab1}
\begin{tabular}{ccc}
\hline    $\Delta t$ ~~& $L^{2}$ error ~~& Order\\[2mm]
\hline    $1/2^6$ ~~& 6.39E-2 ~~& -\\[2mm]
    $1/2^7$ ~~& 1.60E-2 ~~& 1.99\\[2mm]
    $1/2^8$ ~~& 4.00E-3 ~~& 2.02\\[2mm]
    $1/2^9$ ~~& 9.00E-4 ~~& 2.07\\[2mm]
    $1/2^{10}$ ~~& 2.00E-4 ~~& 2.32\\[2mm]
\hline
\end{tabular}
\end{center}
\end{table}
(ii) Let $\lambda=\sqrt{2}$. Table \ref{tab2} presents the mean-square convergence order in time for the proposed method with different sorts of Wiener processes depending on $\mathcal{M}$ along 100 paths.
 The strong order of convergence is approximately 1 for $\mathcal{M}=1,4,8$. It is worth to study the convergence order of the proposed stochastic multi-symplectic wavelet collocation method theoretically.

\begin{table}[htbp]
\begin{center}
\caption{Accuracy test for stochastic Maxwell equations \eqref{stochastic maxwell equations}}
 \label{tab2}
\begin{tabular}{c|cc|cc|cc}
\hline   $\Delta t$ ~~& $L^{2}$ error ~~& Order~~& $L^{2}$ error ~~& Order~~& $L^{2}$ error ~~& Order\\[2mm]
                &  $\mathcal{M}=1$&&$\mathcal{M}=4$&&$\mathcal{M}=8$&\\[2mm]
\hline
    $1/2^6$ ~~&3.03E-1~~&  - ~~&4.63E-1~~&  - ~~&6.51E-1~~&-\\[2mm]
    $1/2^7$ ~~&1.32E-1~~&1.19~~&1.95E-1~~&1.25~~&2.67E-1~~&1.29\\[2mm]
    $1/2^8$ ~~&5.78E-2~~&1.20~~&8.43E-2~~&1.21~~&1.11E-1~~&1.26\\[2mm]
    $1/2^9$ ~~&2.41E-2~~&1.26~~&3.50E-2~~&1.27~~&4.48E-2~~&1.31\\[2mm]
    $1/2^{10}$ ~~&8.20E-3&1.56~~&1.19E-2~~&1.56~~&1.53E-2~~&1.55\\[2mm]
\hline
\end{tabular}
\end{center}
\end{table}

\section{Conclusions}
In this paper, we design an energy-conserving numerical method for 3D stochastic Maxwell equations with multiplicative noise. The method not only solves equations efficiently, but also preserves exactly the discrete stochastic multi-symplectic conservation law and the discrete energy conservation law. Numerical experiments show the good performance of the proposed method. In addition, the mean square convergence order in time is studied numerically. Since the theoretical analysis of the convergence is difficult, we will devote to study it rigorously in the future work. Certainly, it is also very interesting to extend the method to other equations, such as 3D stochastic nonlinear Schr\"{o}dinger equation with multiplicative noise.

\end{document}